\documentclass[12pt]{extarticle}
\usepackage{amsmath, amsthm, amssymb, hyperref, color}
\usepackage{etex}
\usepackage[shortlabels]{enumitem}
\usepackage{graphicx}
\usepackage[all]{xypic}
\usepackage{makecell}
\usepackage[final]{pdfpages}
\setboolean{@twoside}{false}
\usepackage{pdfpages}
\usepackage{caption}
\usepackage{subcaption}
\usepackage{scalefnt}
\usepackage{verbatim}
\usepackage{array}
\usepackage{blkarray}
\usepackage{amsmath, amssymb}
\oddsidemargin \evensidemargin
\newtheorem{theorem}{Theorem}
\numberwithin{theorem}{section}

\newtheorem*{theo}{Main Theorem}
\newtheorem{lemma}[theorem]{Lemma}

\newtheorem{definition}[theorem]{Definition}

\newtheorem{remark}[theorem]{Remark}
\newtheorem{example}[theorem]{Example}
\newtheorem{conjecture}[theorem]{Conjecture}

\newcommand{\ZZ}{\mathbb{Z}}
\newcommand{\NN}{\mathbb{N}}

\def\Gf{\mathfrak{G}}

 \date{}

\usepackage{tikz}
\usepackage{gauss}
\usetikzlibrary{matrix,decorations.pathreplacing,calc}

\pgfkeys{tikz/mymatrixenv/.style={decoration=brace,every left delimiter/.style={xshift=3pt},every right delimiter/.style={xshift=-3pt}}}
\pgfkeys{tikz/mymatrix/.style={matrix of math nodes,left delimiter=[,right delimiter={]},inner sep=2pt,column sep=1em,row sep=0.5em,nodes={inner sep=0pt}}}
\pgfkeys{tikz/mymatrixbrace/.style={decorate,thick}}
\newcommand\mymatrixbraceoffseth{0.5em}
\newcommand\mymatrixbraceoffsetv{0.2em}

\newcommand*\mymatrixbraceright[4][m]{
    \draw[mymatrixbrace] ($(#1.north west)!(#1-#3-1.south west)!(#1.south west)-(\mymatrixbraceoffseth,0)$)
        -- node[left=2pt] {#4} 
        ($(#1.north west)!(#1-#2-1.north west)!(#1.south west)-(\mymatrixbraceoffseth,0)$);
}
\newcommand*\mymatrixbraceleft[4][m]{
    \draw[mymatrixbrace] ($(#1.north east)!(#1-#2-1.north east)!(#1.south east)+(\mymatrixbraceoffseth,0)$)
        -- node[right=2pt] {#4} 
        ($(#1.north east)!(#1-#3-1.south east)!(#1.south east)+(\mymatrixbraceoffseth,0)$);
}

\newcommand*\mymatrixbracebottom[4][m]{
    \draw[mymatrixbrace] ($(#1.south west)!(#1-1-#3.south east)!(#1.south east)-(0,\mymatrixbraceoffsetv)$)
        -- node[below=2pt] {#4} 
        ($(#1.south west)!(#1-1-#2.south west)!(#1.south east)-(0,\mymatrixbraceoffsetv)$);
}

\definecolor{block}{RGB}{0,162,232}

\newenvironment{blockmatrix}{%
  \vcenter\bgroup\hbox\bgroup
  \tikzpicture[
    x=1.5\baselineskip,
    y=1.5\baselineskip,
  ]%
}{%
  \endtikzpicture
  \egroup
  \egroup
}

\newcommand*{\block}[1][block]{%
  \blockaux{#1}%
}
\def\blockaux#1(#2,#3)#4(#5,#6){%
  \draw[fill={#1}]
  let \p1=(#2,#3),
      \p2=(#5,#6),
      \p3=(#2+#5,#3+#6),
      \p4=(#2+#5/2,#3+#6/2)
  in
    (\p1) rectangle (\p3)
    (\p4) node {$#4$}
  ;%
}

\title{\textbf{Finite phylogenetic complexity and combinatorics of tables}}

\author{Mateusz Micha{\l}ek and Emanuele Ventura}

\begin{document}

\maketitle

\begin{abstract}

We prove that the phylogenetic complexity -- an invariant introduced by Sturmfels and Sullivant -- of any finite abelian group is finite.

\end{abstract}

\thanks{2010 \emph{Mathematics Subject Classification}. Primary 52B20, Secondary 14M25, 13P25}

\section{Introduction}

The aim of this article is to prove the finiteness of an intriguing invariant of finite abelian groups, called \emph{phylogenetic complexity}. The invariant was introduced in a seminal paper by Sturmfels and Sullivant \cite{SS}, where it appeared in relation to phylogenetic models. In short, to a Markov process encoded by an abelian group $G$ on a tree $T$ one associates a toric variety $X(G,T)$, of particular relevance in algebraic statistics \cite{4aut, PS}. The setting above is known as a \emph{group--based model}. 

We do not describe the relations to phylogenetics in this paper, referring the interested reader to \cite{allman2003phylogenetic, casanellas2012algebraic, DBM, jaDissert}.
 Instead, in precise, purely mathematical language we present a natural construction of \emph{a family of lattice polytopes $P_{G,n}$} associated to any finite abelian group $G$ -- Definition \ref{def:PGn}. 
These polytopes should be considered as the simpliest combinatorial objects encoding the group action. 

Associating to interesting combinatorial objects a polytope and investigating its properties is nowadays a well-developed and powerful tool on the edge of combinatorics and toric geometry \cite{Stks, herzog2002discrete, ohsugi1998normal, sturmfels2008toric}.  However, our knowledge of properties of the polytopes $P_{G,n}$ associated to such basic objects as finite abelian groups is still very limited. This may be even more surprising, as for various groups $G$, these polytopes relate not only to phylogenetics, but also mathematical physics through conformal blocks and moduli spaces \cite{kubjas2014conformal, Man2, Man3, SX}. 

Phylogenetic complexity governs the degrees of generators of the ideal of the variety $X(G,T)$. Using the language of toric geometry, one is interested in the generators of integral relations among the vertices of $P_{G,n}$. For the introduction to toric geometry we refer the reader to \cite{CLS, Ful}. The objects that encode the group action and correspond to vertices of $P_{G,n}$ are called \emph{flows}.

\begin{definition}[{\bf Flow} \cite{JaAdvGeom}, \cite{BW}]  \label{def:flow}
Let $G$ be a finite abelian group and $n\in \mathbb N$. A flow 
is a sequence of $n$ elements of $G$ summing up to $0\in G$, the neutral element of $G$. The set of flows is equipped with a group structure via the coordinatewise action. 
The group of flows $\Gf$ is (non-canonically) isomorphic to $G^{n-1}$.
\end{definition}
Hence, in our article we study possible relations among $n$-tuples of elements of $G$ summing up to $0$. Let $T_0$ and $T_1$ be two matrices or {\it tables} of the same size, whose rows are flows. These two tables are \emph{compatible} if and only if, for each $1\leq i \leq n$, the $i$-th column of $T_0$ and the $i$-th column of $T_1$ are the same multisets -- cf.~Example \ref{ex:rel}. Compatible tables correspond to binomials in the ideal $I(X(G,K_{1,n}))$, where $K_{1,n}$ is a star (also called a claw-tree) -- the unique tree with one inner vertex and $n$ leaves.

\begin{definition}[{\bf Phylogenetic complexity} \cite{SS}]

Let $T$ be a tree, let $K_{1,n}$ be the star with $n$ leaves, and let $\phi(G,T)$ be the maximal degree of a generator in a minimal generating set of $I(X(G,T))$. Let $\phi(G,n)=\phi(G,K_{1,n})$. We define the phylogenetic complexity $\phi(G)$ of $G$ to be $\sup_{n\in \mathbb N} \phi(G,n)$. 

\end{definition}

The main theorem of the present article is the following:  

\begin{theo}[{\bf Theorem \ref{finite phylo}}]\label{thm:main}

For any finite abelian group $G$, the phylogenetic complexity $\phi(G)$ is finite.

\end{theo}

Let us present the state of the art in the following table. 
\vspace{2mm}
\begin{center}
\begin{tabular}{|p{3,1 cm}|p{1,8 cm}|p{2,3 cm}|p{2,7 cm}|p{4,5 cm}|}
\hline
&\multicolumn{4}{ |c| }{Group-based Models} \\
\hline
Polynomials defining:& $\ZZ_2$ & $\ZZ_3$ & $\ZZ_2\times\ZZ_2$ & $G$ \\
\hline
Generators of the ideal & Degree $2$ \cite{SS} & Degree $3$ \cite{jaZ3nowy} & Conjecture \cite[Conjecture 30]{SS}  & Finite by Theorem \ref{finite phylo}, Degree $\leq |G|$ \cite[Conjecture 29]{SS} \\
\hline
Projective scheme & & Degree $3$ \cite{Marysianowyz3} & Degree $4$ \cite{JaJCTA} & Finite \cite{JaJCTA}\\
\hline
Set-theoretically&&&&Finite \cite{DE}
\\
\hline
On a Zariski open subset & && Degree $4$ \cite{JaAdvGeom} & Degree $\leq|G|$ \cite{CFSM, casanellas2015complete}\\
\hline
\end{tabular}
\end{center}
\vspace{2mm}

A lot of the related results concern finiteness. For equivariant models (which include the class of models described in this article) the finiteness result on set-theoretic level was proved in \cite{DK, DE}. Hence, our paper can be regarded as a stronger result, but for a smaller class. Obtaining finiteness result on an ideal-theoretic level for equivariant models would be a major achievement, far extending the results of \cite{DK2}. However, this is beyond any of the methods described in this paper, where we focus on group-based models. Finiteness plays also an increasingly important role in the context of toric varieties; cf.~\cite{JanNowy}.

Finally, we would like to mention the reduction that we use from the very beginning, previously obtained by Sturmfels and Sullivant \cite{SS}. Although, in general, one is interested in arbitrary trees, it is enough to consider claw-trees. This is due to the construction of toric fiber products \cite{Sethtfp}.

The structure of the article is as follows. In Section \ref{sec:notation} we describe the basic notation. In particular, we recall how one encodes binomials in $I(X(G,K_{1,n}))$ as special pairs of tables with group elements. Section \ref{sec:proof} contains the main result. First, in Subsection \ref{subs:sketch}, we present the sketch of the proof, without any technical details and then the complete proof in Subsection \ref{subs:proof}. We hope that some of the ideas of the paper can be made effective. In particular, in future work we plan to prove \cite[Conjecture 30]{SS}.

\section{Binomials, Tables and Moves}\label{sec:notation}

This section records definitions and notation needed in the rest of the paper.\\
\indent  Let $G$ be a finite abelian group and $n\in\NN$. In Definition \ref{def:flow}, we introduced the most important algebro-combinatorial objects in our setting: $n$-tuples of group elements summing to $0$, called \emph{flows}. From the point of view of toric geometry and phylogenetics, flows correspond to \emph{monomials} parameterizing our variety $X(G , K_{1,n})$ \cite{SS, JaJalg}. Relations among flows -- described by \emph{compatible} tables -- encode the binomials in $I(X(G, K_{1,n}))$. It is a standard approach in toric geometry to represent the parameterizing monomials by their exponents, as points in a lattice. The polytope, that is the convex hull of such points, captures the geometry of the parameterized variety. For the sake of completeness we present the polytopes corresponding to $X(G , K_{1,n})$.

\begin{definition}[{\bf Polytope $P_{G,n}$}]\label{def:PGn}
Consider the lattice $M\cong \ZZ^{|G|}$ with a basis corresponding to elements of $G$. Consider $M^n$ with the basis $e_{(i,g)}$ indexed by pairs $(i,g)\in [n]\times G$. We define an injective map of sets:
$\Gf\rightarrow M^n,$
by $(g_1,\dots,g_n)\longmapsto \sum_{i=1}^n e_{(i,g_i)}$. The image of this map defines the vertices of the polytope $P_{G,n}$.
\end{definition}

\begin{example}[{\bf \cite{jaZ3nowy}}]

For $G=(\ZZ_2,+)$ and $n=3$, we have four flows:
$$(0,0,0),(0,1,1),(1,0,1),(1,1,0)\in \ZZ_2\times \ZZ_2\times \ZZ_2.$$

\noindent Hence, the polytope $P_{\ZZ_2,3}$ has the following four vertices corresponding to the flows above:
$$(1,0,1,0,1,0),(1,0,0,1,0,1),(0,1,1,0,0,1),(0,1,0,1,1,0)\in\ZZ^2\times \ZZ^2\times \ZZ^2,$$
where $(1,0)\in \ZZ^2$ corresponds to $0\in \ZZ_2$ and $(0,1)\in\ZZ^2$ corresponds to $1\in \ZZ_2$.
\end{example}

A more sophisticated example is presented in \cite[Example 4.1]{JaJalg}. Binomials may be identified with a pair of tables of the same size $T_0$ and $T_1$ of elements of $G$, regarded up to row permutation. Each row of such tables has to be a flow. The identification is as follows. Every binomial is a pair of monomials; the variables in such monomials correspond to flows, given by a collection of $n$ elements in $G$. Every monomial is viewed as a table, whose rows are the variables appearing in the monomial; the number of rows of the corresponding table is the degree of the monomial. Consequently, a binomial is identified with the pair of tables encoding the two monomials respectively. \\
\indent A binomial belongs to $I(X(G , K_{1,n}))$ if and only if the two tables are \emph{compatible}, i.e.~for each $i$, the $i$-th column of $T_0$ and the $i$-th column of $T_1$ are equal as multisets.\\ 
\indent In order to generate a binomial -- represented by a pair of tables $T_0$, $T_1$ -- by binomials of degree at most $d$ we are allowed to select a subset of rows in $T_0$ of cardinality at most $d$ and replace it with a compatible set of rows, repeating this procedure until both tables are equal. 

\begin{example}[{\bf \cite{jaZ3nowy}}]\label{ex:rel}

For $G=(\ZZ_2,+)$ and $n=6$ consider the following two compatible tables:
$$
T_0=\begin{bmatrix}
\color{red}{1} & \color{red}{1} & \color{red}{1} &  \color{red}{1} & \color{red}{1} & \color{red}{1} \\
\color{red}{0} & \color{red}{0} & \color{red}{0} & \color{red}{0} & \color{red}{0} & \color{red}{0} \\
\color{brown}{1} & \color{brown}{1} & \color{brown}{0} & \color{brown}{0} & \color{brown}{0} & \color{brown}{0}\\
\end{bmatrix}
\textnormal{ and   } 
T_1=\begin{bmatrix}
\color{blue}{0} & \color{blue}{1} & \color{blue}{0} & \color{blue}{1} & \color{blue}{0}& \color{blue}{0} \\
1 & 0 & 1 & 0 & 0 & 0 \\ 
1 & 1 & 0 & 0 & 1 & 1 \\
\end{bmatrix}.
$$
\noindent 
Note that the red subtable of $T_0$ is compatible with the table 
$$
T'=\begin{bmatrix}
\color{blue}{0} & \color{blue}{1} & \color{blue}{0} & \color{blue}{1} & \color{blue}{0}& \color{blue}{0} \\
\color{brown}{1} & \color{brown}{0} & \color{brown}{1} & \color{brown}{0} & \color{brown}{1} & \color{brown}{1}\\
\end{bmatrix}.
$$
\noindent Hence, we may \emph{exchange} them obtaining:
$$
\tilde T_0=\begin{bmatrix}
\color{blue}{0} & \color{blue}{1} & \color{blue}{0} & \color{blue}{1} & \color{blue}{0}& \color{blue}{0} \\
\color{brown}{1} & \color{brown}{0} & \color{brown}{1} & \color{brown}{0} & \color{brown}{1} & \color{brown}{1}\\ 
\color{brown}{1} & \color{brown}{1} & \color{brown}{0} & \color{brown}{0} & \color{brown}{0} & \color{brown}{0}\\
\end{bmatrix}.
$$

\noindent Note that $T_0$ and $\tilde T_0$ are \emph{compatible}.  Now, the brown subtable of $\tilde T_0$ is compatible with the table
$$
T''=\begin{bmatrix}
1 & 0 & 1 & 0 & 0 & 0 \\
1 & 1 & 0 & 0 & 1 & 1 \\
\end{bmatrix}.
$$
\noindent Finally, we exchange them obtaining $T_1$. Hence we have a sequence of tables $T_0 \rightsquigarrow \tilde T_0 \rightsquigarrow T_1$. More specifically, we started from a degree three binomial given by the pair $T_0, T_1$ and we generated it using degree two binomials, called {\it quadratic moves}; see also Example \ref{exquadmove}. 

\end{example}

In what follows, \emph{quadratic moves}, i.e.~binomials of degree two will play a crucial role. First, let us give the precise definition and an illustrative example.

\begin{definition}[{\bf Quadratic Moves}]\label{def:QM}

Let $T$ be a table -- whose rows are flows -- of elements of $G$; let $r_i$ and $r_j$ be two rows of $T$. For any subsequence $\lbrace r_{i,l_1},\ldots r_{i,l_t}\rbrace$ of $r_i$, we define two rows $s_i$ and $s_j$ whose elements are the following: 
\begin{enumerate}
\item[(i)] $s_{i,k}=r_{i,k}$ if $k\neq l_1,\ldots, l_t$, otherwise $s_{i,k}=r_{j,k}$; 
\item[(ii)] $s_{j,k}=r_{j,k}$ if $k\neq l_1,\ldots, l_t$, otherwise $s_{j,k}=r_{i,k}$. 

\end{enumerate}
\noindent The transformation of $r_i$ and $r_j$ into $s_i$ and $s_j$ described above is a {\it quadratic move} if $\sum_{k=1}^t r_{i, l_k} = \sum_{k=1}^t r_{j, l_k}$; in other words, if the differences sum up to $0\in G$. We note that this condition is equivalent to the fact that $s_i$ and $s_j$ are flows.

\end{definition}

To illustrate the definition of quadratic moves, we consider the following example, to be compared with Example \ref{ex:rel}. 

\begin{example}\label{exquadmove}

Let $G=(\mathbb Z_2, +)$. Let $T$ be the following $2 \times 3$ table of elements in $\mathbb Z_2$: 
$$
T=\begin{bmatrix}
{\color{red} 1} & 1 & {\color{red} 0}  & 1 & 1 \\
{\color{blue} 0}  & 1 & {\color{blue} 1}  & 0  & 0 \\
\end{bmatrix}.
$$
\noindent The two rows $r_1$ and $r_2$ are flows, since their elements sum up to the $0\in \ZZ_2$. We exchange the red subsequence of elements in the first row with the blue subsequence of elements in the second row. The rows $s_1$ and $s_2$, corresponding to the chosen (red) subsequence as in Definition \ref{def:QM}, are the two rows of the following table: 

$$
\tilde T=\begin{bmatrix}
0 & 1 & 1 & 1 & 1 \\
1 & 1 & 0  & 0  & 0 \\
\end{bmatrix}.
$$
\noindent This is a quadratic move, since $s_1$ and $s_2$ are still flows. Hence, the table $T$ is transformed into the table $\tilde T$ by the quadratic move above. Note that quadratic moves preserve, up to permutation, each column of a table. In particular, $T$ and $\tilde T$ are two compatible tables, i.e. their columns are the same as multisets. 

\end{example}

\section{Finite phylogenetic complexity for abelian groups}\label{sec:proof}

The aim of this section is to use the combinatorics of tables to prove finiteness of the phylogenetic complexity of a group-based model for any finite abelian group $G$. 

\subsection{Idea of the proof}\label{subs:sketch}
Before going into technical details, let us present here the basic ideas of Theorem \ref{finite phylo}. 

The general strategy is to prove that the function $\phi(G,n)$ is eventually constant for large $n$. Hence, we start with two compatible $d\times n$ tables $T_0$ and $T_1$ for large $n$ and we want to transform $T_0$ to $T_1$. The main objective is the proof of Lemma \ref{Two columns with same rows}:
\emph{One can transform $T_0$ and $T_1$, independently, using quadratic moves, in such a way that there exist two columns $c_j$, $c_{j+1}$ on which both tables exactly agree.}
Once this aim is achieved, the induction becomes clear -- the precise argument is presented in the last paragraph of the proof of Theorem \ref{finite phylo}. 
The most involved part is to show Lemma \ref{Two columns with same rows}. First, we pass to subtables. For a table $T$, we denoted by $T'$ the subtable containing all rows, but only those columns where a given element $g\in G$ is one of the (possibly many) most frequent group elements. This is not a severe restriction
 --  cf.~Remark \ref{most freq. el.}. Such a `reference' element $g$ is crucial throughout the proof. Note also that, due to compatibility, the  indices of columns of the subtables $T_0'$ and $T_1'$ are the same, as the most frequent elements of any $i$-th column in $T_0$ and $T_1$ coincide. In particular, $T_0'$ and $T_1'$ are compatible (although their rows do not have to be flows any more). In the proof, it is shown that it is easier to move elements that are {\it frequent} in a table, than those that are {\it rare}; the latter ones are  called \emph{dots}. A precise definition, independent on the choice of $T_0$ or $T_1$, of frequent and rare elements is given in Definition \ref{def:dots}.\\
\indent  Equipped with these definitions, it is enough to prove Lemma \ref{Two columns not containing dots in same row}:
\emph{One can transform $T_0$ and $T_1$, independently, using quadratic moves, in such a way that there exist two columns $c_j$, $c_{j+1}$ such that any row in $T_0$ or $T_1$ contains at most one dot in columns $c_j$ and $c_{j+1}$.}
Indeed, once the above statement is proven, as the tables are considered up to row permutations, we can make all dots in both columns in $T_0$ exactly equal to corresponding dots in $T_1$. Then Lemma \ref{Two columns with same rows} follows as the entries that are not dots can also be adjusted -- details are in the proof of the lemma. 

Hence, the hard part of the proof of Theorem \ref{finite phylo} lies in the proof of Lemma \ref{Two columns not containing dots in same row}. Here the ideas are as follows. First, (as we passed from $T$ to a subtable $T'$ where a given element is one of most frequent in every column) we will be passing to thinner and thinner subtables. However, due to technical reasons, we must also allow their horizontal subdivisions, which motivate the following definition.

\begin{definition}[Vertical stripe]
Given any table $T$, we define a \emph{vertical stripe} to be:
\begin{itemize}
\item a choice of some number of consecutive columns of $T$,
\item a subdivision of rows into parts in the chosen columns.
\end{itemize}
Less formally, a vertical stripe is a collection of disjoint subtables in the same columns, that cover all rows of $T$.
\end{definition}
Two examples of vertical stripes are presented in Figure (\ref{fig:em}). One consists of the whole colored part, where the subdivision into three subtables is given by two thick white horizontal stripes. The second stripe is the yellow one with the subdivision into nine subtables.

We would like to find a vertical stripe with (at least) two columns that has at most one dot in each row. Instead, we consider more general subtables that make vertical stripes: each subtable has $k$ columns with at most $s$ dots in each row. Further, we need to control how many distinct elements $r$ of $G$ appear as dots in the subtable. These subtables do not have to contain all rows, but appear in collections that form a vertical stripe, i.e.~the collection covers all rows.\\

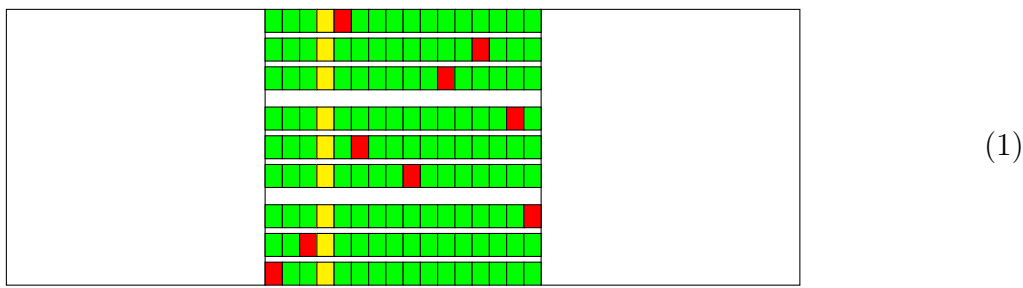
\begin{equation}\label{fig:em}
\begin{blockmatrix}

\block[white] (-4.5, 1.9) (4.5, 4.8)

\block[white] (4.8, 1.9) (4.5, 4.8)

 \block[green](0, 6.3) (0.3, 0.4)
\block[green](0.3, 6.3) (0.3,0.4)
    \block[green](0.6, 6.3) (0.3,0.4)
    \block[yellow](0.9, 6.3) (0.3,0.4)
    \block[red](1.2, 6.3) (0.3,0.4)
    \block[green](1.5, 6.3) (0.3, 0.4)
    \block[green](1.8, 6.3) (0.3,0.4)
      \block[green](2.1, 6.3) (0.3,0.4)
    \block[green](2.4, 6.3) (0.3,0.4)
      \block[green](2.7, 6.3) (0.3,0.4)
    \block[green](3.0, 6.3) (0.3,0.4)
          \block[green](3.3, 6.3) (0.3,0.4)
    \block[green](3.6,  6.3) (0.3,0.4)
    \block[green](3.9,  6.3) (0.3,0.4)
    \block[green](4.2,  6.3) (0.3,0.4)
    \block[green](4.5,  6.3) (0.3,0.4)

   \block[green](0, 5.8) (0.3, 0.4)
\block[green](0.3, 5.8) (0.3,0.4)
    \block[green](0.6, 5.8) (0.3,0.4)
    \block[yellow](0.9, 5.8) (0.3,0.4)
    \block[green](1.2, 5.8) (0.3,0.4)
    \block[green](1.5, 5.8) (0.3, 0.4)
    \block[green](1.8, 5.8) (0.3,0.4)
      \block[green](2.1, 5.8) (0.3,0.4)
    \block[green](2.4, 5.8) (0.3,0.4)
      \block[green](2.7, 5.8) (0.3,0.4)
    \block[green](3.0, 5.8) (0.3,0.4)
          \block[green](3.3, 5.8) (0.3,0.4)
    \block[red](3.6,  5.8) (0.3,0.4)
    \block[green](3.9,  5.8) (0.3,0.4)
    \block[green](4.2,  5.8) (0.3,0.4)
    \block[green](4.5,  5.8) (0.3,0.4)

  \block[green](0, 5.3) (0.3, 0.4)
\block[green](0.3, 5.3) (0.3,0.4)
    \block[green](0.6, 5.3) (0.3,0.4)
    \block[yellow](0.9, 5.3) (0.3,0.4)
    \block[green](1.2, 5.3) (0.3,0.4)
    \block[green](1.5, 5.3) (0.3, 0.4)
    \block[green](1.8, 5.3) (0.3,0.4)
      \block[green](2.1, 5.3) (0.3,0.4)
    \block[green](2.4, 5.3) (0.3,0.4)
      \block[green](2.7, 5.3) (0.3,0.4)
    \block[red](3.0, 5.3) (0.3,0.4)
          \block[green](3.3, 5.3) (0.3,0.4)
    \block[green](3.6,  5.3) (0.3,0.4)
    \block[green](3.9,  5.3) (0.3,0.4)
    \block[green](4.2,  5.3) (0.3,0.4)
    \block[green](4.5,  5.3) (0.3,0.4)

  \block[green](0, 4.6) (0.3, 0.4)
\block[green](0.3, 4.6) (0.3,0.4)
    \block[green](0.6, 4.6) (0.3,0.4)
    \block[yellow](0.9, 4.6) (0.3,0.4)
    \block[green](1.2, 4.6) (0.3,0.4)
    \block[green](1.5, 4.6) (0.3, 0.4)
    \block[green](1.8, 4.6) (0.3,0.4)
      \block[green](2.1, 4.6) (0.3,0.4)
    \block[green](2.4, 4.6) (0.3,0.4)
      \block[green](2.7, 4.6) (0.3,0.4)
    \block[green](3.0, 4.6) (0.3,0.4)
          \block[green](3.3, 4.6) (0.3,0.4)
    \block[green](3.6,  4.6) (0.3,0.4)
    \block[green](3.9,  4.6) (0.3,0.4)
    \block[red](4.2,  4.6) (0.3,0.4)
    \block[green](4.5,  4.6) (0.3,0.4)

  \block[green](0, 4.1) (0.3, 0.4)
\block[green](0.3, 4.1) (0.3,0.4)
    \block[green](0.6, 4.1) (0.3,0.4)
    \block[yellow](0.9, 4.1) (0.3,0.4)
    \block[green](1.2, 4.1) (0.3,0.4)
    \block[red](1.5, 4.1) (0.3, 0.4)
    \block[green](1.8, 4.1) (0.3,0.4)
      \block[green](2.1, 4.1) (0.3,0.4)
    \block[green](2.4, 4.1) (0.3,0.4)
      \block[green](2.7, 4.1) (0.3,0.4)
    \block[green](3.0, 4.1) (0.3,0.4)
          \block[green](3.3, 4.1) (0.3,0.4)
    \block[green](3.6,  4.1) (0.3,0.4)
    \block[green](3.9,  4.1) (0.3,0.4)
    \block[green](4.2,  4.1) (0.3,0.4)
    \block[green](4.5,  4.1) (0.3,0.4)

   \block[green](0, 3.6) (0.3, 0.4)
\block[green](0.3, 3.6) (0.3,0.4)
    \block[green](0.6, 3.6) (0.3,0.4)
    \block[yellow](0.9, 3.6) (0.3,0.4)
    \block[green](1.2, 3.6) (0.3,0.4)
    \block[green](1.5, 3.6) (0.3, 0.4)
    \block[green](1.8, 3.6) (0.3,0.4)
      \block[green](2.1, 3.6) (0.3,0.4)
    \block[red](2.4, 3.6) (0.3,0.4)
      \block[green] (2.7, 3.6) (0.3,0.4)
    \block[green](3.0, 3.6) (0.3,0.4)
          \block[green](3.3, 3.6) (0.3,0.4)
    \block[green](3.6,  3.6) (0.3,0.4)
    \block[green](3.9,  3.6) (0.3,0.4)
    \block[green](4.2,  3.6) (0.3,0.4)
    \block[green](4.5,  3.6) (0.3,0.4)

  \block[green](0, 2.9) (0.3, 0.4)
\block[green](0.3, 2.9) (0.3,0.4)
    \block[green](0.6, 2.9) (0.3,0.4)
    \block[yellow](0.9, 2.9) (0.3,0.4)
    \block[green](1.2, 2.9) (0.3,0.4)
    \block[green](1.5, 2.9) (0.3, 0.4)
    \block[green](1.8, 2.9) (0.3,0.4)
      \block[green](2.1, 2.9) (0.3,0.4)
    \block[green](2.4, 2.9) (0.3,0.4)
      \block[green](2.7, 2.9) (0.3,0.4)
    \block[green](3.0, 2.9) (0.3,0.4)
          \block[green](3.3, 2.9) (0.3,0.4)
    \block[green](3.6,  2.9) (0.3,0.4)
    \block[green](3.9,  2.9) (0.3,0.4)
    \block[green](4.2,  2.9) (0.3,0.4)
    \block[red](4.5,  2.9) (0.3,0.4)

    \block[green](0, 2.4) (0.3,0.4)
    \block[green](0.3, 2.4) (0.3,0.4)
    \block[red](0.6, 2.4) (0.3,0.4)
    \block[yellow](0.9, 2.4) (0.3,0.4)
    \block[green](1.2, 2.4) (0.3,0.4)
    \block[green](1.5, 2.4) (0.3, 0.4)
    \block[green](1.8, 2.4) (0.3,0.4)
      \block[green](2.1, 2.4) (0.3,0.4)
    \block[green](2.4, 2.4) (0.3,0.4)
      \block[green](2.7, 2.4) (0.3,0.4)
    \block[green](3.0, 2.4) (0.3,0.4)
          \block[green](3.3, 2.4) (0.3,0.4)
    \block[green](3.6,  2.4) (0.3,0.4)
    \block[green](3.9,  2.4) (0.3,0.4)
    \block[green](4.2,  2.4) (0.3,0.4)
    \block[green](4.5,  2.4) (0.3,0.4)

  \block[red](0, 1.9) (0.3, 0.4)
    \block[green](0.3, 1.9) (0.3,0.4)
    \block[green](0.6, 1.9) (0.3,0.4)
    \block[yellow](0.9, 1.9) (0.3,0.4)
    \block[green](1.2, 1.9) (0.3,0.4)
    \block[green](1.5, 1.9) (0.3,0.4)
    \block[green](1.8, 1.9) (0.3, 0.4 )
      \block[green](2.1, 1.9 ) (0.3,0.4)
    \block[green](2.4, 1.9) (0.3, 0.4)
      \block[green](2.7, 1.9 ) (0.3, 0.4)
    \block[green](3.0, 1.9 ) (0.3, 0.4)
          \block[green](3.3, 1.9) (0.3, 0.4)
    \block[green](3.6,  1.9) (0.3, 0.4)
    \block[green](3.9,  1.9) (0.3,0.4)
    \block[green](4.2,  1.9) (0.3,0.4)
    \block[green](4.5,  1.9) (0.3, 0.4)

  \end{blockmatrix}
\end{equation}
\noindent The table above pictures the subdivision algorithm deviced in Lemma \ref{Two columns with same rows} for $T_0'$ and $T_1'$. We start with a vertical stripe -- here represented by the colored part of the table. It consists of three subtables, divided by two large horizontal white stripes.
In each of the subtables, we fix the same partition of columns into $t=16$ vertical and three horizontal parts (given numbers are just examples). The new, finer horizontal subdivision is depicted with thin white stripes.  
In each horizontal part, we discard at most one of the subtables -- these are the red squares. The yellow part drawn in the center of the picture is a vertical stripe, consisting of subtables that are not discarded in any of the horizontal parts. \\ 
\vspace{5mm}

The main point is that, for large $k$, we may decrease $s$ or $r$ by subdividing each subtable into $t |G|$ small subtables: columns are divided into $t\gg 0$ parts and rows into $|G|$ parts. In particular, we have $t$ vertical parts, each consisting of $|G|$ small subtables stacked one under another. After quadratic moves, we may assume that each small subtable in almost all of the $t$ vertical parts either has smaller number of dots in each row (decreasing $s$) or smaller number of distinct group elements corresponding to dots (decreasing $r$). As $t$ is always much greater than the number of horizontal subdivisions (which is always some power of $|G|$) we are able to choose a whole vertical stripe (with much smaller number of columns) such that in each subtable $s$ or $r$ has been decreased. Further, we are able to do it in parallel in $T_0'$ and $T_1'$ -- details are in the proof of Lemma \ref{Two columns not containing dots in same row}.\\
\indent We hope this discussion could shed some light on Definition \ref{good table}. We mention here a technical remark: since we work with vertical stripes, once we focus on one subtable, we have to make sure we {\it do not} change the structure of other subtables. This feature is reflected in (ii) of Definition \ref{good table}, where we restrict to quadratic moves that only modify a small part of the table.  
We are finally able to list the main steps towards the proof of Lemma \ref{Two columns not containing dots in same row}:

\begin{enumerate}
\item[(i)] bound the number of dots in each row (Lemma \ref{dots in each row});
\item[(ii)] prove that we may always subdivide a subtable, as described above, decreasing $s$ or $r$ (Lemma \ref{gen coloring});
\item[(iii)] show that the subdivision process can be done in parallel in $T_0$ and $T_1$ (Lemma \ref{Two columns not containing dots in same row}).
\end{enumerate}
\subsection{Proof}\label{subs:proof}
We start from the definition of frequent elements in a given table $T$ with respect to a function $F$. Let $F(G)$ be a function of the cardinality of the group $G$. We assume $F(G)>| G |^2+3|G|$.

\begin{definition}[{\bf Set $F_T$, Dots}]\label{def:dots}

The set of $F(G)$-{\it frequent elements}, or frequent elements, in a given $d\times n$ table $T$ is defined by 

$$
F_{T}=\lbrace h\in G | \mbox{ number of copies of } h \mbox{ in } T > F(G)\cdot d\rbrace. 
$$

\noindent  
Note that if an element is frequent, then there exists a row, where it appears at least $F(G)$ times. The elements $g\in G$ that are not in $F_{T}$ are called {\it dots} $\bullet$. 

\end{definition}

The frequent elements have a key role in allowing quadratic moves in the table. 
Let us start with three basic -- yet useful -- lemmas. 

\begin{lemma}\label{lem:move}

Let $f,f'$ be flows. Let $I$ be a subset of indices and suppose $|I|\geq |G|$. 
There exists a (non-empty) subset $I'\subset I$ such that a quadratic move of $f$ and $f'$ on $I'$ can be performed.

\end{lemma}

\begin{proof}

Since we have $|G|$ differences, possibly repeated, of the form $f_i-f'_i$ for $i\in I$, we may find a non-empty subset $I'$, such that $\sum_{i\in I'}(f_i-f'_i)=0\in G$.
\end{proof}

\begin{lemma}\label{dots in each row}

Let $T$ be a given table of elements of $G$, then we may assume that each row in $T$ has at most $|G|(F(G)+1)$ dots. 

\begin{proof}

Note that there exists a row containing at most $|G|F(G)$ dots. Assuming the contrary, we would have  at least $(|G|F(G)+1)d$ dots in $T$. This would imply that there would be a dot in $F_T$ -- a contradiction. Let us consider a row $r_{\max}$ with the largest number of dots. If $r_{\max}$ contains at most $|G|(F(G)+1)$ dots, this finishes the proof. Otherwise, 
we pick a row $r_{\min}$ with the smallest number of dots; they are at most $|G|F(G)$. Now, there exist $|G|$ dots of $r_{\max}$ in the same columns as $|G|$ elements of $r_{\min}$ which are in $F_T$. Exchanging a subset of them we decrease the number of rows with the largest number of dots. 
Repeating the process, we obtain $T$ with all rows with at most $|G|(F(G)+1)$ dots. 
\end{proof}
\end{lemma}

\begin{lemma}\label{Counting of zeros}

Let $z\in \mathbb N$. For any $\epsilon > 0$, there exists $n=n(z)$ such that in any $(0,1)$-table $T$ of size $d\times n$, whose columns contain at least $\epsilon \cdot d$ zeros each, there exists a row with at least $z$ zeros. 

\begin{proof}
Setting $n >z/\epsilon$ we may conclude by double counting zeros column and row-wise. 
\end{proof}
\end{lemma}

\begin{remark}\label{most freq. el.}

Let $T$ be a $d\times n$ table whose entries are elements of $G$. 
In each column $c_i$, we select the elements that appear a maximal number of times; these elements are the {\it most frequent elements} in $c_i$. Among all the columns, we select those where a {\it reference} element $g\in G$ appears as one of the most frequent elements. \\
\indent This is not a severe restriction, as $n$ is very large and we would restrict to a subtable with at least $n/|G|$ columns, for some $g\in G$. Such a reference element $g$ will be important throughout the proof. 
\end{remark}
We now introduce a crucial property $S(\cdot)$ for our inductive proof.
\begin{definition}[{\bf Property $S(\cdot)$}]\label{good table}

Let $s,r,t,k\in \mathbb N$, let  $T$ be a $d\times n$ table whose entries are elements of $G$, and $Q$ a $d' \times k$-subtable of $T$. Moreover, let us assume that the following holds:

\begin{itemize}

\item[(a)] $g\in G$ is one of the most frequent elements in every column of $T$;

\item[(b)] there are at most $s$ dots in every row of $Q$;

\item[(c)] there exists a subset $H\subset G$ of cardinality $r$, such that each dot of $Q$ belongs to $H$.

\end{itemize}
 
\noindent We say that the property $S(s,r,t,k,T,Q)$ holds for the pair $Q\subset T$ if

\begin{enumerate}

\item[(i)] $s=1$ and $k\geq 2$, or 

\item[(ii)] $s>1$ and we can transform $T$ into another table $\tilde T$ (transforming $Q$ into $\tilde Q$) such that:

we may subdivide the first $t\cdot \left \lfloor{k/t}\right \rfloor$ columns of $\tilde Q$ 
into $t$ consecutive  subtables $Q_i$, each consisting of $\tilde k = \left \lfloor{k/t}\right \rfloor$ columns and $d'$ rows that satisfy:
\subitem (1) $r=1$: Each of the $Q_i$'s {\it except one} has the property $S(s-1,r, | G |t , \tilde k, T, Q_i)$. 
\subitem (2) $r>1$: For every $Q_i$ {\it except one} we can subdivide the rows into $|G|$ parts $Q_{ij}$, such that for every $j$ either $S(s-1,r, | G |t , \tilde k, T, Q_{ij})$ or $S(s, r-1, | G |t, \tilde k, T, Q_{ij})$ holds.  
Further, the transformation may only use quadratic moves that do not change dots that are in the columns of $Q$ and in rows outside $Q$ (i.e.~it cannot move dots in the same vertical stripe, but outside $Q$).
\end{enumerate}
\begin{remark}
Condition (a) in Definition \ref{good table} is not restrictive, according to Remark \ref{most freq. el.}, as we will be applying the definition to subtables of $T_0$ and $T_1$ for which $g$ is one of the most frequent elements in each column.
\end{remark}

\end{definition}
In the next Lemma, we show that one can transform and divide $Q$ into smaller subtables decreasing either $s$ or $r$, provided $k$ is sufficiently large. This is achieved with special quadratic moves.

\begin{lemma}\label{gen coloring}

For every $s,r,t\in \mathbb N$, for every $k$ sufficiently large, for every pair $Q\subset T$ satisfying the assumptions in Definition \ref{good table}, the property $S(s,r,t,k,T,Q)$ holds.

\begin{proof}

The proof is by induction on $s$. For $s=1$ the claim is true for $k\geq 2$  by Definition \ref{good table}. Let us assume that the claim is true for $s$. We show the statement for $s+1$. \\
\indent If $s+1 \geq | G |$, let us set $k > t\cdot \tilde k$, where $\tilde k$ is an integer such that the property $S(s, r, | G | t, \tilde k, \cdot , \cdot)$ holds for arbitrary pairs of tables and subtables in the last two arguments (satisfying assumptions in Definition \ref{good table}). Let us fix an arbitrary pair of tables $Q\subset T$ satisfying assumptions in Definition \ref{good table}. By Definition \ref{good table}, 
each row of $Q$ has at most $s+1$ dots. We fix a partition of $Q$ into equal-sized subtables $Q_j$, each consisting of $\lfloor k/t\rfloor$ consecutive columns.
If all $Q_j$s contain only rows with strictly less than $s+1$ dots, we are done. Otherwise, we choose a subtable $Q_{i_0}$ with a maximal 
number of rows containing $s+1$ dots. 
Every $Q_j$ has at most as many rows with $s+1$ dots as $Q_{i_0}$. Hence, for any subtable $Q_j$  different from $Q_{i_0}$ we can pair each row of $Q_j$ with  $s+1$ dots with a row of $Q_j$ without any dots (the latter corresponding to a row of $Q_{i_0}$ with $s+1$ dots). The structure of $T$ is as follows:

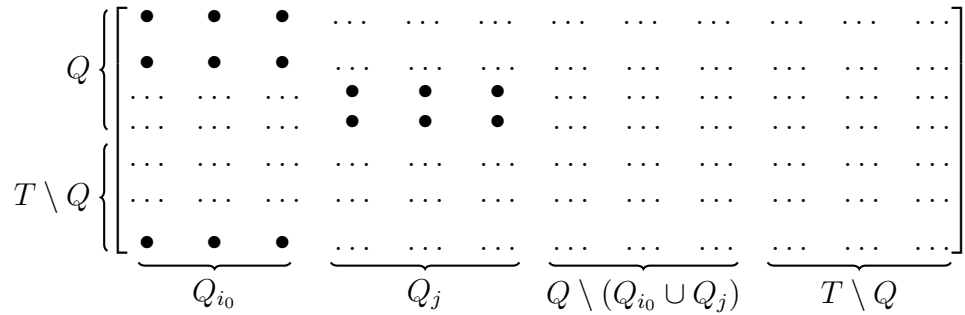
\[
\begin{tikzpicture}[mymatrixenv]
\matrix[mymatrix] (m)  {
 \bullet & \bullet & \bullet & \ldots & \ldots & \ldots & \ldots & \ldots & \ldots & \ldots & \ldots & \ldots \\
  &  &  &  &  &  &  &  &   &   &  &  \\
\bullet & \bullet & \bullet & \ldots & \ldots  & \ldots & \ldots & \ldots & \ldots & \ldots & \ldots & \ldots\\
\ldots & \ldots & \ldots & \bullet & \bullet & \bullet  & \ldots & \ldots & \ldots & \ldots & \ldots & \ldots\\
\ldots & \ldots & \ldots & \bullet & \bullet & \bullet   & \ldots & \ldots & \ldots & \ldots & \ldots & \ldots\\
\vspace{1mm}\\
\ldots & \ldots & \ldots & \ldots & \ldots & \ldots  & \ldots & \ldots & \ldots & \ldots & \ldots & \ldots \\
\vspace{1mm}\\
\ldots & \ldots & \ldots & \ldots & \ldots & \ldots  & \ldots & \ldots & \ldots & \ldots & \ldots & \ldots \\
\vspace{1mm}\\
\bullet & \bullet & \bullet & \ldots & \ldots & \ldots  & \ldots & \ldots & \ldots & \ldots & \ldots & \ldots \\
    };

    \mymatrixbraceright{1}{5}{$Q$}
    \mymatrixbraceright{6}{11}{$T\setminus Q$}
    \mymatrixbracebottom{1}{3}{$Q_{i_0}$}
    \mymatrixbracebottom{4}{6}{$Q_j$}
    \mymatrixbracebottom{7}{9}{$Q\setminus (Q_{i_0} \cup Q_j)$}
    \mymatrixbracebottom{10}{12}{$T\setminus Q$}
\end{tikzpicture}
\]

\noindent The arrows below describe the pairing between a row with  $s+1$ dots with a row without any dots in the subtable $Q_j$. \\

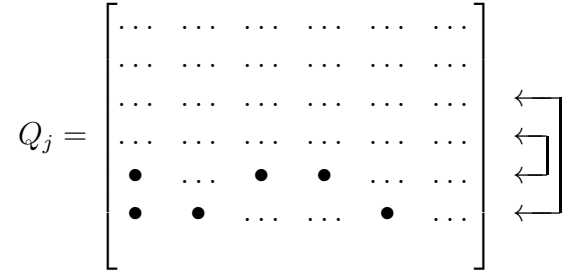
$$
Q_j=\begin{gmatrix}[b]
\ldots & \ldots & \ldots & \ldots & \ldots & \ldots \\
\ldots & \ldots & \ldots & \ldots & \ldots & \ldots\\
\ldots & \ldots & \ldots & \ldots & \ldots & \ldots\\
\ldots & \ldots & \ldots & \ldots & \ldots & \ldots\\
\bullet & \ldots &\bullet & \bullet & \ldots & \ldots\\
\bullet &\bullet &\ldots  & \ldots  & \bullet & \ldots \\ 
 \rowops
 \swap{3}{4}
 \swap{2}{5}
\end{gmatrix}
$$

\indent For each such pair, we make a quadratic move reducing the number of dots that a row of $Q_j$ may have.  Hence, by induction, for any $Q_j\neq Q_{i_0}$ the property $S(s,r,| G |t, \left \lfloor{k/t}\right \rfloor, T, Q_j)$ holds, as $k > t \cdot \tilde{k}$. Thus $S(s+1,r, t, k, T, Q)$ holds by Definition \ref{good table}. 

If $s+1< | G | $, we proceed by induction on $r$. 

If $r=1$, let us set $k > t\cdot \tilde k$ as before. First, suppose there is only one vertical part $Q_{i_0}$ which contains rows with $s+1$ dots. Since all the other parts $Q_j$'s have rows with at most $s$ dots, by induction they satisfy 
$S(s,r,| G | t, \left \lfloor{k/t}\right \rfloor, T, Q_j)$, hence we may conclude in this case. Otherwise, as long as there are two parts $Q_{i_0}$ and $Q_{j_0}$ with rows $r_i$ and $r_j$ respectively with $s+1$ dots, we proceed as follows. Let us fix one dot in $r_i$ and one in $r_j$. Let $g_i$ and $g_j$ be the elements of the rows $r_i$ and $r_j$ in the same columns as the chosen dots. 

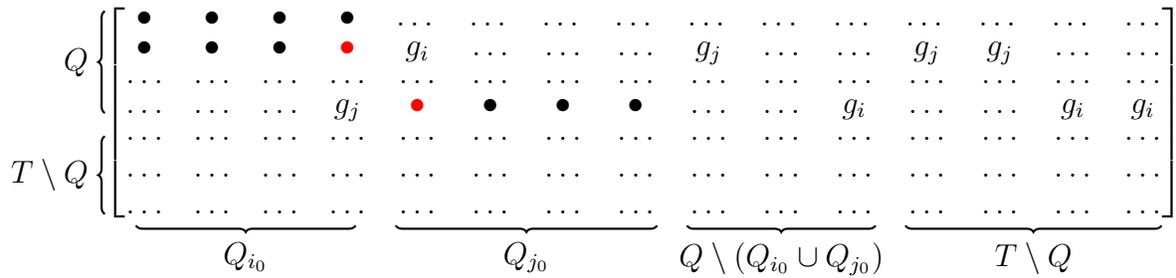
\[
\begin{tikzpicture}[mymatrixenv]
\matrix[mymatrix] (m)  {
\bullet & \bullet & \bullet &\bullet & \ldots & \ldots & \ldots &\ldots &\ldots &\ldots & \ldots & \ldots &\ldots &\ldots & \ldots \\
\bullet & \bullet & \bullet & \color{red}{\bullet} & g_i & \ldots & \ldots &\ldots &g_j  &\ldots & \ldots & g_j & g_j & \ldots & \ldots \\
\ldots & \ldots  & \ldots & \ldots  & \ldots & \ldots & \ldots &\ldots  &\ldots &\ldots & \ldots &\ldots &\ldots & \ldots  &\ldots \\
\ldots & \ldots & \ldots & g_j & \color{red}{\bullet} & \bullet & \bullet & \bullet &\ldots &\ldots & g_i & \ldots & \ldots & g_i & g_i \\
\ldots & \ldots & \ldots & \ldots & \ldots & \ldots & \ldots & \ldots & \ldots &\ldots & \ldots &\ldots &\ldots & \ldots &\ldots\\
\vspace{1mm} \\
\ldots & \ldots & \ldots & \ldots & \ldots & \ldots & \ldots & \ldots & \ldots &\ldots & \ldots &\ldots &\ldots & \ldots &\ldots\\
\vspace{1mm} \\
\ldots & \ldots & \ldots & \ldots & \ldots & \ldots & \ldots & \ldots & \ldots &\ldots & \ldots &\ldots &\ldots & \ldots &\ldots\\
    };
    \mymatrixbraceright{1}{4}{$Q$}
    \mymatrixbraceright{5}{9}{$T\setminus Q$}
    \mymatrixbracebottom{1}{4}{$Q_{i_0}$}
    \mymatrixbracebottom{5}{8}{$Q_{j_0}$}
    \mymatrixbracebottom{9}{11}{$Q\setminus (Q_{i_0}\cup Q_{j_0})$}
    \mymatrixbracebottom{12}{15}{$T\setminus Q$}
\end{tikzpicture}
\] 

\noindent As $g_i$ is not a dot, there has to exist a row $r_t$ of $T$ with more than $F(G)$ copies of $g_i$. We make a quadratic move between $r_j$ and $r_t$ not involving the $2s+2$ columns of dots in $Q_{i_0}$ and $Q_{j_0}$ in the rows $r_i$ and $r_j$. This procedure allows us to put at least $F(G)-3 | G |$ copies of $g_i$ in the row $r_j$, without moving dots in $Q$ -- we need to subtract $|G|$ by Lemma \ref{lem:move} and $2|G|>2(s+1)$ to avoid the dots. Now, we can make the same quadratic move for $g_j$ and $r_i$. The result of these moves is in the table above, where the red bullets $\color{red}{\bullet}$ are the chosen dots. \\
\indent After performing these quadratic moves, if there is a column $c_t$ containing $g_j$ and $g_i$ in rows $r_i$ and $r_j$, then we make a quadratic move, exchanging the chosen dots and the elements of $c_t$. Otherwise, applying Lemma \ref{Counting of zeros} for $\epsilon=1/|G|$ to a subtable of $T$ of columns containing $g_i$ in the row $r_j$, we may find a row $r_t$ containing at least $ | G |$ copies of $g$, as long as $F(G)-3|G|>|G|^2$. Then we move some copies of $g$ to the row $r_i$ by Lemma \ref{lem:move}. Analogously for $g_j$, we may move some copies of $g$ to the row $r_j$. Here are depicted in red the copies of $g$ and in blue the quadratic move putting those copies of $g$ in $r_i$ and $r_j$ respectively.\\

\[
\begin{blockarray}{ccccccccccc}
 &  &  &  &  & & & & & & \\
\begin{block}{[cccccccccc] c}
\ldots & \ldots & \ldots & \ldots & \ldots & \ldots & {\color{red} g} & {\color{red} g} & {\color{red} g} & {\color{red} g}  & r_t\\
\ldots & \ldots & \ldots & \ldots & \ldots & \ldots & \color{blue}{\downarrow} & \color{blue}{\downarrow} & \color{blue}{\downarrow} & \color{blue}{\downarrow} &  \\
{\color{red} \bullet} & g_i & g_j  & g_j & \ldots & g_j & \ldots & \ldots & \ldots & \ldots & r_i\\
\ldots & \ldots & \ldots & \ldots & \ldots & \ldots & \ldots & \ldots & \ldots & \ldots &  \\
g_j & {\color{red} \bullet} & \ldots & \ldots & \ldots & \ldots & g_i & g_i & \ldots  & g_i  & r_j\\
\ldots & \ldots & \color{blue}{\uparrow} & \color{blue}{\uparrow} & \color{blue}{\uparrow} & \color{blue}{\uparrow} & \ldots & \ldots & \ldots & \ldots &  \\
\ldots & \ldots & {\color{red} g} & {\color{red} g} & {\color{red} g} & {\color{red} g} & \ldots & \ldots & \ldots & \ldots & r_{t'}\\
\end{block}
\end{blockarray}
\]

\noindent Applying the blue quadratic move above, we obtain a column $c_i$ that has $g$ in $r_i$ and $g_i$ in $r_j$. In the same way, we obtain a column $c_j$ that has $g$ in $r_j$ and $g_j$ in $r_i$. Now, we perform a quadratic move in the subtable below, exchanging the chosen dots:

$$
\begin{bmatrix}
\color{red}{\bullet} & g_i & g_j & g  \\
g_j & \color{red}{\bullet} &  g & g_i \\
\end{bmatrix}.
$$

Thus, we reduce the number of dots in both rows. This concludes the case $r=1$. \\
\indent  Assume $r > 1$. Let us set $k > t\cdot \tilde k$, where $\tilde k$ is such that both of the properties $S(s,r,| G |t, k, T, \cdot)$ and 
$S(s+1,r-1,| G |t, k, T, \cdot)$ hold. Suppose that there is only one $Q_{i_0}$ such that there exists a row with $s+1$ dots corresponding to $r$ distinct group elements. Then the rows of every other part $Q_j$ can be partitioned into at most $| G |$ parts $Q_{j,l}$, such that 

\begin{itemize}

\item[(i)] all the rows  in $Q_{j,1}$  have at most $s$ dots;

\item[(ii)] all the dots in $Q_{j,l}$ for $l > 1$ correspond to at most $r-1$ distinct group elements. 

\end{itemize}

We conclude by induction in the case when there is only one part $Q_{i_0}$. We will reduce every other case to this one. Assume that there are two parts $Q_{i_0}$ and $Q_{j_0}$ such that there exist rows $r_i$ and $r_j$ with $s+1$ dots corresponding to $r$ distinct group elements. 
As both rows $r_i$ and $r_j$ contain dots corresponding to the same $r$ elements of the group $G$, we can choose one dot in each row corresponding to the same element. Now, repeating the procedure described in the case $r=1$, we reduce the number of dots in $r_i$ and $r_j$. This concludes the proof. 
\end{proof}

\end{lemma}
By Lemma \ref{gen coloring} for any $s,r,t$ we set $K(s,r,t)$ such that for all $k\geq K(s,r,t)$ the property $S(s,r,t,k,\cdot,\cdot)$ holds.

\begin{lemma}\label{Two columns not containing dots in same row}

Let $T_0$ and $T_1$ be two compatible tables with at least $|G|K(|G|(F(G)+1),|G|,3)$ columns. Then, we may transform them using quadratic moves into two tables $\tilde T_0$ and $\tilde T_1$ such that the following holds: 
there exists $j$ such that no row in $\tilde T_0$ nor in $\tilde T_1$ has a dot in both the $j$-th and the $(j+1)$-st columns.  

\begin{proof}

Let us restrict $T_0$ and $T_1$ to the subtables $T'_0$ and $T'_1$ containing all rows and those columns that have $g$ as the most frequent element. By Lemma \ref{dots in each row}, we may assume that the upper bound on the number of dots in $T'_0$ and $T'_1$ in each row is $B=|G|(F(G)+1)$. By Remark \ref{most freq. el.} and the assumption on the size of $T_0$ and $T_1$, we can assume that $T'_0$ and $T'_1$ have at least $k_0=K(B,|G|,3)$ columns. Hence, in particular, the properties $S(B, | G |, 3, k_0, T_0', T'_0)$ and $S(B, | G |, 3, k_0, T_1', T'_1)$ hold. In the rest of the proof we transform both tables $T_0'$ and $T_1'$ using quadratic moves, at each step passing to a smaller vertical stripe such that each subtable in it satisfies the property $S(\cdot)$ with smaller and smaller $s$ or $r$. \\
\indent We apply the following algorithm, which is depicted in Figure (\ref{fig:em}). The starting point of the $i$-th step of the algorithm are two compatible tables with corresponding distinguished $ k_i=\lfloor{ k_{i-1}/(3|G|^{i-1})}\rfloor$ consecutive columns forming a vertical stripe. In the $i$-th step, the vertical stripe has at most $|G|^i$ parts (subtables). In the table $T'_0$ the parts are $T'_{0,j}$. For a given part $T'_{0,j}$, let $s_{0,i,j}$ be the maximal number of dots that a row may have. Let $r_{0,i,j}$ be the number of distinct group elements correspoding to dots of $T'_{0,j}$. Then $S(s_{0,i,j},r_{0,i,j}, 3 |G|^i, k_i, T'_0, T'_{0,j})$ holds. Likewise the parts $T'_{1,j}$ of $T'_1$ satisfy $S(s_{1,i,j},r_{1,i,j}, 3 |G|^i, k_i, T'_1, T'_{1,j})$. Let us subdivide the $k_i$ columns into $3 |G|^i$ parts, subdividing each $T'_{0,j}$ into parts $T'_{0,j,a}$, as in Definition \ref{good table}. 
Now, the algorithm transforms $T'_{0,j,a}$ using Definition \ref{good table}. Hence, we obtain a subdivision of rows of $T'_{0,j,a}$ into at most $| G |$ parts $T'_{0,j,a,b}$. Here are the parts of $T'_0$ highlighted in blue, where the left and right brackets select horizontal parts and the bottom bracket selects vertical parts:

\[
\begin{tikzpicture}[mymatrixenv]
\matrix[mymatrix] (m)  {
\ldots & \ldots & \ldots  & \ldots  & \ldots  \\
\ldots & \ldots & \ldots  & \ldots  & \ldots \\
\ldots & \ldots & \ldots  & \ldots  & \ldots \\
\ldots & \ldots & \ldots  & \ldots  & \ldots \\
\ldots & \ldots & \ldots  & \ldots  & \ldots \\
\ldots & \ldots & \ldots  & \ldots  & \ldots \\
\ldots & \ldots & \ldots  & \ldots  & \ldots \\
\ldots & \ldots & \ldots  & \ldots  & \ldots \\
\ldots & \ldots & \ldots  & \ldots  & \ldots \\
\ldots & \ldots & \ldots  & \ldots  & \ldots \\
\ldots & \ldots & \ldots  & \ldots  & \ldots \\
};

    \mymatrixbraceright{1}{7}{$\color{blue}{T'_{0,j}}$}
    \mymatrixbraceleft{1}{3}{$\color{blue}{T'_{0,j,a,b}}$}
    \mymatrixbracebottom{1}{2}{$\color{blue}{T'_{0,j,a}}$}
    
\end{tikzpicture}
\]

For each $j$, for every $a$ {\it except one} and for every $b$ the subtable $T'_{0,j,a,b}$ satisfies either the property $S(s_{0,i,j}-1, r_{0,i,j}, 3|G|^{i+1},k_{i+1}, T'_0, T'_{0,j,a,b})$ or $S(s_{0,i,j}, r_{0,i,j}-1, 3|G|^{i+1},k_{i+1}, T'_0, T'_{0,j,a,b})$. As each of the $|G|^{i}$ horizontal parts in $T_0'$ can exclude one $T'_{0,j,a}$, and each of the $|G|^{i}$ horizontal parts in $T_1'$ can exclude one $T'_{1,j,a}$ we may find an index $a_0$ such that: 
\begin{enumerate}
\item[(i)] $S(s_{0,i,j}-1, r_{0,i,j}, 3|G|^{i+1},k_{i+1}, T'_0, T'_{0,j,a_0,b})$ or $S(s_{0,i,j}, r_{0,i,j}-1, 3|G|^{i+1},k_{i+1}, T'_0, T'_{0,j,a_0,b})$ 
\end{enumerate}
\indent and
\begin{enumerate}
\item[(ii)] $S(s_{1,i,j}-1, r_{1,i,j}, 3|G|^{i+1},k_{i+1}, T'_1, T'_{1,j,a_0,b})$ or $S(s_{1,i,j}, r_{1,i,j}-1, 3|G|^{i+1},k_{i+1}, T'_1, T'_{1,j,a_0,b})$ 
\end{enumerate}

hold for every $j$ and every $b$. (Less formally, since the number of vertical stripes is much larger than the number of discarded subtables in each subdivision, we can choose two corresponding vertical stripes in both of the tables. This is pictured in Figure (\ref{fig: em2}).) The choice of the $a_0$-th vertical stripe and the subdivisions $T'_{0,j,a_0,b}$, $T'_{1,j,a_0,b}$ are the output of the $i$-th step of the algorithm and the input of the $(i+1)$-th step. The algorithm terminates when we reach $s=1$. The procedure terminates in a finite number of steps as at each step either $s$ or $r$ decreases. Moreover, at every step of the algorithm, we have  collections of subtables satisfying property $S(\cdot)$. This implies that, at the last step, $k\geq 2$. Thus the algorithm provides the desired pairs of columns.

\begin{equation}\label{fig: em2}
\tilde{T_0}-\tilde{T_1}= \begin{tiny} \begin{blockmatrix}
\block[white] (-4.5, 1.9) (4.5, 4.8)
\block[white] (4.8, 1.9) (4.5, 4.8)
\block[green](0, 6.3) (0.3, 0.4)
\block[green](0.3, 6.3) (0.3,0.4)
    \block[green](0.6, 6.3) (0.3,0.4)
    \block[yellow](0.9, 6.3) (0.3,0.4)
    \block[red](1.2, 6.3) (0.3,0.4)
    \block[green](1.5, 6.3) (0.3, 0.4)
    \block[green](1.8, 6.3) (0.3,0.4)
      \block[green](2.1, 6.3) (0.3,0.4)
    \block[green](2.4, 6.3) (0.3,0.4)
      \block[green](2.7, 6.3) (0.3,0.4)
    \block[green](3.0, 6.3) (0.3,0.4)
          \block[green](3.3, 6.3) (0.3,0.4)
    \block[green](3.6,  6.3) (0.3,0.4)
    \block[green](3.9,  6.3) (0.3,0.4)
    \block[green](4.2,  6.3) (0.3,0.4)
    \block[green](4.5,  6.3) (0.3,0.4)

   \block[green](0, 5.8) (0.3, 0.4)
\block[green](0.3, 5.8) (0.3,0.4)
    \block[green](0.6, 5.8) (0.3,0.4)
    \block[yellow](0.9, 5.8) (0.3,0.4)
    \block[green](1.2, 5.8) (0.3,0.4)
    \block[green](1.5, 5.8) (0.3, 0.4)
    \block[green](1.8, 5.8) (0.3,0.4)
      \block[green](2.1, 5.8) (0.3,0.4)
    \block[green](2.4, 5.8) (0.3,0.4)
      \block[green](2.7, 5.8) (0.3,0.4)
    \block[green](3.0, 5.8) (0.3,0.4)
          \block[green](3.3, 5.8) (0.3,0.4)
    \block[red](3.6,  5.8) (0.3,0.4)
    \block[green](3.9,  5.8) (0.3,0.4)
    \block[green](4.2,  5.8) (0.3,0.4)
    \block[green](4.5,  5.8) (0.3,0.4)

  \block[green](0, 5.3) (0.3, 0.4)
\block[green](0.3, 5.3) (0.3,0.4)
    \block[green](0.6, 5.3) (0.3,0.4)
    \block[yellow](0.9, 5.3) (0.3,0.4)
    \block[green](1.2, 5.3) (0.3,0.4)
    \block[green](1.5, 5.3) (0.3, 0.4)
    \block[green](1.8, 5.3) (0.3,0.4)
      \block[green](2.1, 5.3) (0.3,0.4)
    \block[green](2.4, 5.3) (0.3,0.4)
      \block[green](2.7, 5.3) (0.3,0.4)
    \block[red](3.0, 5.3) (0.3,0.4)
          \block[green](3.3, 5.3) (0.3,0.4)
    \block[green](3.6,  5.3) (0.3,0.4)
    \block[green](3.9,  5.3) (0.3,0.4)
    \block[green](4.2,  5.3) (0.3,0.4)
    \block[green](4.5,  5.3) (0.3,0.4)

  \block[green](0, 4.6) (0.3, 0.4)
\block[green](0.3, 4.6) (0.3,0.4)
    \block[green](0.6, 4.6) (0.3,0.4)
    \block[yellow](0.9, 4.6) (0.3,0.4)
    \block[green](1.2, 4.6) (0.3,0.4)
    \block[green](1.5, 4.6) (0.3, 0.4)
    \block[green](1.8, 4.6) (0.3,0.4)
      \block[green](2.1, 4.6) (0.3,0.4)
    \block[green](2.4, 4.6) (0.3,0.4)
      \block[green](2.7, 4.6) (0.3,0.4)
    \block[green](3.0, 4.6) (0.3,0.4)
          \block[green](3.3, 4.6) (0.3,0.4)
    \block[green](3.6,  4.6) (0.3,0.4)
    \block[green](3.9,  4.6) (0.3,0.4)
    \block[red](4.2,  4.6) (0.3,0.4)
    \block[green](4.5,  4.6) (0.3,0.4)

  \block[green](0, 4.1) (0.3, 0.4)
\block[green](0.3, 4.1) (0.3,0.4)
    \block[green](0.6, 4.1) (0.3,0.4)
    \block[yellow](0.9, 4.1) (0.3,0.4)
    \block[green](1.2, 4.1) (0.3,0.4)
    \block[red](1.5, 4.1) (0.3, 0.4)
    \block[green](1.8, 4.1) (0.3,0.4)
      \block[green](2.1, 4.1) (0.3,0.4)
    \block[green](2.4, 4.1) (0.3,0.4)
      \block[green](2.7, 4.1) (0.3,0.4)
    \block[green](3.0, 4.1) (0.3,0.4)
          \block[green](3.3, 4.1) (0.3,0.4)
    \block[green](3.6,  4.1) (0.3,0.4)
    \block[green](3.9,  4.1) (0.3,0.4)
    \block[green](4.2,  4.1) (0.3,0.4)
    \block[green](4.5,  4.1) (0.3,0.4)

   \block[green](0, 3.6) (0.3, 0.4)
\block[green](0.3, 3.6) (0.3,0.4)
    \block[green](0.6, 3.6) (0.3,0.4)
    \block[yellow](0.9, 3.6) (0.3,0.4)
    \block[green](1.2, 3.6) (0.3,0.4)
    \block[green](1.5, 3.6) (0.3, 0.4)
    \block[green](1.8, 3.6) (0.3,0.4)
      \block[green](2.1, 3.6) (0.3,0.4)
    \block[red](2.4, 3.6) (0.3,0.4)
      \block[green] (2.7, 3.6) (0.3,0.4)
    \block[green](3.0, 3.6) (0.3,0.4)
          \block[green](3.3, 3.6) (0.3,0.4)
    \block[green](3.6,  3.6) (0.3,0.4)
    \block[green](3.9,  3.6) (0.3,0.4)
    \block[green](4.2,  3.6) (0.3,0.4)
    \block[green](4.5,  3.6) (0.3,0.4)

  \block[green](0, 2.9) (0.3, 0.4)
\block[green](0.3, 2.9) (0.3,0.4)
    \block[green](0.6, 2.9) (0.3,0.4)
    \block[yellow](0.9, 2.9) (0.3,0.4)
    \block[green](1.2, 2.9) (0.3,0.4)
    \block[green](1.5, 2.9) (0.3, 0.4)
    \block[green](1.8, 2.9) (0.3,0.4)
      \block[green](2.1, 2.9) (0.3,0.4)
    \block[green](2.4, 2.9) (0.3,0.4)
      \block[green](2.7, 2.9) (0.3,0.4)
    \block[green](3.0, 2.9) (0.3,0.4)
          \block[green](3.3, 2.9) (0.3,0.4)
    \block[green](3.6,  2.9) (0.3,0.4)
    \block[green](3.9,  2.9) (0.3,0.4)
    \block[green](4.2,  2.9) (0.3,0.4)
    \block[red](4.5,  2.9) (0.3,0.4)

    \block[green](0, 2.4) (0.3,0.4)
    \block[green](0.3, 2.4) (0.3,0.4)
    \block[red](0.6, 2.4) (0.3,0.4)
    \block[yellow](0.9, 2.4) (0.3,0.4)
    \block[green](1.2, 2.4) (0.3,0.4)
    \block[green](1.5, 2.4) (0.3, 0.4)
    \block[green](1.8, 2.4) (0.3,0.4)
      \block[green](2.1, 2.4) (0.3,0.4)
    \block[green](2.4, 2.4) (0.3,0.4)
      \block[green](2.7, 2.4) (0.3,0.4)
    \block[green](3.0, 2.4) (0.3,0.4)
          \block[green](3.3, 2.4) (0.3,0.4)
    \block[green](3.6,  2.4) (0.3,0.4)
    \block[green](3.9,  2.4) (0.3,0.4)
    \block[green](4.2,  2.4) (0.3,0.4)
    \block[green](4.5,  2.4) (0.3,0.4)

  \block[red](0, 1.9) (0.3, 0.4)
    \block[green](0.3, 1.9) (0.3,0.4)
    \block[green](0.6, 1.9) (0.3,0.4)
    \block[yellow](0.9, 1.9) (0.3,0.4)
    \block[green](1.2, 1.9) (0.3,0.4)
    \block[green](1.5, 1.9) (0.3,0.4)
    \block[green](1.8, 1.9) (0.3, 0.4 )
      \block[green](2.1, 1.9 ) (0.3,0.4)
    \block[green](2.4, 1.9) (0.3, 0.4)
      \block[green](2.7, 1.9 ) (0.3, 0.4)
    \block[green](3.0, 1.9 ) (0.3, 0.4)
          \block[green](3.3, 1.9) (0.3, 0.4)
    \block[green](3.6,  1.9) (0.3, 0.4)
    \block[green](3.9,  1.9) (0.3,0.4)
    \block[green](4.2,  1.9) (0.3,0.4)
    \block[green](4.5,  1.9) (0.3, 0.4)
  \end{blockmatrix}
 -
\begin{blockmatrix}
\block[white] (-4.5, 1.9) (4.5, 4.8)
\block[white] (4.8, 1.9) (4.5, 4.8)
\block[green](0, 6.3) (0.3, 0.4)
\block[green](0.3, 6.3) (0.3,0.4)
    \block[green](0.6, 6.3) (0.3,0.4)
    \block[yellow](0.9, 6.3) (0.3,0.4)
    \block[green](1.2, 6.3) (0.3,0.4)
    \block[green](1.5, 6.3) (0.3, 0.4)
    \block[green](1.8, 6.3) (0.3,0.4)
      \block[green](2.1, 6.3) (0.3,0.4)
    \block[green](2.4, 6.3) (0.3,0.4)
      \block[green](2.7, 6.3) (0.3,0.4)
    \block[green](3.0, 6.3) (0.3,0.4)
          \block[green](3.3, 6.3) (0.3,0.4)
    \block[green](3.6,  6.3) (0.3,0.4)
    \block[green](3.9,  6.3) (0.3,0.4)
    \block[green](4.2,  6.3) (0.3,0.4)
    \block[red](4.5,  6.3) (0.3,0.4)

   \block[green](0, 5.8) (0.3, 0.4)
\block[green](0.3, 5.8) (0.3,0.4)
    \block[green](0.6, 5.8) (0.3,0.4)
    \block[yellow](0.9, 5.8) (0.3,0.4)
    \block[green](1.2, 5.8) (0.3,0.4)
    \block[green](1.5, 5.8) (0.3, 0.4)
    \block[green](1.8, 5.8) (0.3,0.4)
      \block[green](2.1, 5.8) (0.3,0.4)
    \block[green](2.4, 5.8) (0.3,0.4)
      \block[green](2.7, 5.8) (0.3,0.4)
    \block[green](3.0, 5.8) (0.3,0.4)
          \block[green](3.3, 5.8) (0.3,0.4)
    \block[red](3.6,  5.8) (0.3,0.4)
    \block[green](3.9,  5.8) (0.3,0.4)
    \block[green](4.2,  5.8) (0.3,0.4)
    \block[green](4.5,  5.8) (0.3,0.4)

  \block[green](0, 5.3) (0.3, 0.4)
\block[green](0.3, 5.3) (0.3,0.4)
    \block[green](0.6, 5.3) (0.3,0.4)
    \block[yellow](0.9, 5.3) (0.3,0.4)
    \block[green](1.2, 5.3) (0.3,0.4)
    \block[green](1.5, 5.3) (0.3, 0.4)
    \block[green](1.8, 5.3) (0.3,0.4)
      \block[red](2.1, 5.3) (0.3,0.4)
    \block[green](2.4, 5.3) (0.3,0.4)
      \block[green](2.7, 5.3) (0.3,0.4)
    \block[green](3.0, 5.3) (0.3,0.4)
          \block[green](3.3, 5.3) (0.3,0.4)
    \block[green](3.6,  5.3) (0.3,0.4)
    \block[green](3.9,  5.3) (0.3,0.4)
    \block[green](4.2,  5.3) (0.3,0.4)
    \block[green](4.5,  5.3) (0.3,0.4)

  \block[green](0, 4.6) (0.3, 0.4)
\block[green](0.3, 4.6) (0.3,0.4)
    \block[green](0.6, 4.6) (0.3,0.4)
    \block[yellow](0.9, 4.6) (0.3,0.4)
    \block[green](1.2, 4.6) (0.3,0.4)
    \block[green](1.5, 4.6) (0.3, 0.4)
    \block[green](1.8, 4.6) (0.3,0.4)
      \block[green](2.1, 4.6) (0.3,0.4)
    \block[red](2.4, 4.6) (0.3,0.4)
      \block[green](2.7, 4.6) (0.3,0.4)
    \block[green](3.0, 4.6) (0.3,0.4)
          \block[green](3.3, 4.6) (0.3,0.4)
    \block[green](3.6,  4.6) (0.3,0.4)
    \block[green](3.9,  4.6) (0.3,0.4)
    \block[green](4.2,  4.6) (0.3,0.4)
    \block[green](4.5,  4.6) (0.3,0.4)

  \block[green](0, 4.1) (0.3, 0.4)
\block[green](0.3, 4.1) (0.3,0.4)
    \block[green](0.6, 4.1) (0.3,0.4)
    \block[yellow](0.9, 4.1) (0.3,0.4)
    \block[green](1.2, 4.1) (0.3,0.4)
    \block[red](1.5, 4.1) (0.3, 0.4)
    \block[green](1.8, 4.1) (0.3,0.4)
      \block[green](2.1, 4.1) (0.3,0.4)
    \block[green](2.4, 4.1) (0.3,0.4)
      \block[green](2.7, 4.1) (0.3,0.4)
    \block[green](3.0, 4.1) (0.3,0.4)
          \block[green](3.3, 4.1) (0.3,0.4)
    \block[green](3.6,  4.1) (0.3,0.4)
    \block[green](3.9,  4.1) (0.3,0.4)
    \block[green](4.2,  4.1) (0.3,0.4)
    \block[green](4.5,  4.1) (0.3,0.4)

   \block[green](0, 3.6) (0.3, 0.4)
\block[green](0.3, 3.6) (0.3,0.4)
    \block[red](0.6, 3.6) (0.3,0.4)
    \block[yellow](0.9, 3.6) (0.3,0.4)
    \block[green](1.2, 3.6) (0.3,0.4)
    \block[green](1.5, 3.6) (0.3, 0.4)
    \block[green](1.8, 3.6) (0.3,0.4)
      \block[green](2.1, 3.6) (0.3,0.4)
    \block[green](2.4, 3.6) (0.3,0.4)
      \block[green] (2.7, 3.6) (0.3,0.4)
    \block[green](3.0, 3.6) (0.3,0.4)
          \block[green](3.3, 3.6) (0.3,0.4)
    \block[green](3.6,  3.6) (0.3,0.4)
    \block[green](3.9,  3.6) (0.3,0.4)
    \block[green](4.2,  3.6) (0.3,0.4)
    \block[green](4.5,  3.6) (0.3,0.4)

  \block[green](0, 2.9) (0.3, 0.4)
\block[green](0.3, 2.9) (0.3,0.4)
    \block[green](0.6, 2.9) (0.3,0.4)
    \block[yellow](0.9, 2.9) (0.3,0.4)
    \block[green](1.2, 2.9) (0.3,0.4)
    \block[green](1.5, 2.9) (0.3, 0.4)
    \block[green](1.8, 2.9) (0.3,0.4)
      \block[green](2.1, 2.9) (0.3,0.4)
    \block[green](2.4, 2.9) (0.3,0.4)
      \block[green](2.7, 2.9) (0.3,0.4)
    \block[green](3.0, 2.9) (0.3,0.4)
          \block[green](3.3, 2.9) (0.3,0.4)
    \block[green](3.6,  2.9) (0.3,0.4)
    \block[green](3.9,  2.9) (0.3,0.4)
    \block[red](4.2,  2.9) (0.3,0.4)
    \block[green](4.5,  2.9) (0.3,0.4)

    \block[green](0, 2.4) (0.3,0.4)
    \block[green](0.3, 2.4) (0.3,0.4)
    \block[green](0.6, 2.4) (0.3,0.4)
    \block[yellow](0.9, 2.4) (0.3,0.4)
    \block[red](1.2, 2.4) (0.3,0.4)
    \block[green](1.5, 2.4) (0.3, 0.4)
    \block[green](1.8, 2.4) (0.3,0.4)
      \block[green](2.1, 2.4) (0.3,0.4)
    \block[green](2.4, 2.4) (0.3,0.4)
      \block[green](2.7, 2.4) (0.3,0.4)
    \block[green](3.0, 2.4) (0.3,0.4)
          \block[green](3.3, 2.4) (0.3,0.4)
    \block[green](3.6,  2.4) (0.3,0.4)
    \block[green](3.9,  2.4) (0.3,0.4)
    \block[green](4.2,  2.4) (0.3,0.4)
    \block[green](4.5,  2.4) (0.3,0.4)

  \block[green](0, 1.9) (0.3, 0.4)
    \block[red](0.3, 1.9) (0.3,0.4)
    \block[green](0.6, 1.9) (0.3,0.4)
    \block[yellow](0.9, 1.9) (0.3,0.4)
    \block[green](1.2, 1.9) (0.3,0.4)
    \block[green](1.5, 1.9) (0.3,0.4)
    \block[green](1.8, 1.9) (0.3, 0.4 )
      \block[green](2.1, 1.9 ) (0.3,0.4)
    \block[green](2.4, 1.9) (0.3, 0.4)
      \block[green](2.7, 1.9 ) (0.3, 0.4)
    \block[green](3.0, 1.9 ) (0.3, 0.4)
          \block[green](3.3, 1.9) (0.3, 0.4)
    \block[green](3.6,  1.9) (0.3, 0.4)
    \block[green](3.9,  1.9) (0.3,0.4)
    \block[green](4.2,  1.9) (0.3,0.4)
    \block[green](4.5,  1.9) (0.3, 0.4)
  \end{blockmatrix}
\end{tiny}
\end{equation}

\end{proof}
\end{lemma}

\begin{lemma}\label{Two columns with same rows}

Let $T_0$ and $T_1$ be two compatible tables with $n$ columns, for $n$ sufficiently large. Then we can transform $T_0$ and $T_1$ using quadratic moves such that the following holds: there exists $j$ such that the $j$-th and $(j+1)$-st columns in $T_0$ equal respectively the $j$-th and $(j+1)$-st columns in $T_1$. 

\begin{proof}

We restrict to subtables $T'_0$ and $T'_1$ where $g$ is the most frequent element, as in Remark \ref{most freq. el.}. By Lemma \ref{Two columns not containing dots in same row}, we may assume that in every row of the $T'_0$ and $T'_1$ we have only one dot in the first two columns. Now, we can permute rows in such a way that the dots are equal in the corresponding entries. The elements in the rows which are not dots are not necessarily the same in each row. We show that given any pair of elements $g_i, g_j\in F_{T'_0}$ in the first column and in the rows $r_i, r_j$ respectively, we can exchange them.  

Since $g_i$ and $g_j$ are in $F_{T'_0}$, we can find two rows, say $r_s$ and $r_t$ respectively, such that we have at least $F(G)$ copies of $g_i$ and $F(G)$ copies of $g_j$ in $r_s$ and $r_t$ respectively -- see the table below. 
By Lemma \ref{lem:move}, we can move at least $F(G)-|G|-2$ copies of $g_i$ to the row $r_j$ and at least $F(G)-|G|-2$ copies of $g_j$ to $r_i$; here we subtract two because we are avoiding the first two columns. If there is a column $c_t$ containing $g_i$ and $g_j$ in its $j$-th and $i$-th rows respectively, then we exchange them by a quadratic move on the column $c_t$ and the first column. Otherwise, we proceed as follows. We restrict to a subtable containing columns where the row $r_j$ has $g_i$ as its entries. By Lemma \ref{Counting of zeros} for $\epsilon=1/|G|$, in this subtable we may find $|G|$ copies of $g$ in some row $r_t$.  Then we move some copies of $g$ to the row $r_i$ applying Lemma \ref{lem:move}. Analogously for $g_j$, we may move some copies of $g$ to the row $r_j$. Here are depicted in red the copies of $g$ and in blue the quadratic moves putting those copies of $g$ in $r_i$ and $r_j$ respectively.

$$
T'_0=\begin{bmatrix}
\bullet & \ldots & \ldots & \ldots &\ldots & \ldots & \ldots & \ldots & \ldots & \ldots \\
\bullet & \ldots & \color{red}{g} & \color{red}{g} & \color{red}{g} & \color{red}{g} & \ldots & \ldots & \ldots & \ldots \\
\bullet & \ldots & \color{blue}{\downarrow} & \color{blue}{\downarrow} & \color{blue}{\downarrow} &\color{blue}{\downarrow} & \ldots  & \ldots & \ldots & \ldots \\
g_i & \ldots & \ldots & \ldots & \ldots & \ldots  & g_j & g_j & g_j & g_j \\
\ldots & \ldots & \ldots & \ldots & \ldots & \ldots & \ldots & \ldots & \ldots & \ldots \\
\ldots & \bullet & \ldots & \ldots & \ldots & \ldots & \ldots & \ldots & \ldots & \ldots \\
g_j & \bullet &  g_i & g_i & g_i & g_i & \ldots & \ldots & \ldots & \ldots \\
\ldots & \bullet & \ldots & \ldots & \ldots & \ldots &\color{blue}{\uparrow} & \color{blue}{\uparrow} & \color{blue}{\uparrow} & \color{blue}{\uparrow} \\
\ldots & \ldots & \ldots & \ldots & \ldots & \ldots & \color{red}{g} & \color{red}{g} & \color{red}{g} & \color{red}{g} \\
\ldots & \ldots & \ldots & \ldots & \ldots & \ldots & \ldots & \ldots &\ldots &\ldots \\
\end{bmatrix}.
$$
Now, we perform a quadratic move exchanging $g_i$ and $g_j$ in a suitable subtable of $T_0'$:

$$
\begin{bmatrix}
g_i & g_j & g  \\
g_j &  g & g_i \\
\end{bmatrix}.
$$
\noindent Such moves allow to adjust all elements in the first two columns that are {\it not} dots. This concludes the proof. 
\end{proof}

\end{lemma}

\begin{theorem}\label{finite phylo}

For any finite abelian group $G$, the phylogenetic complexity $\phi(G)$ of $G$ is finite. 

\begin{proof}

Let $G$ be a finite abelian group. Fix $N\gg |G|$. Once $N$ is fixed, the phylogenetic complexity $\phi(G,N)$ is finite by the Hilbert Basis Theorem. Assume $n>N$. We will show that $\phi(G,n)\leq \phi(G,n-1)$. This implies that they are equal. \\
\indent Let $B$ be a binomial in $I(X(G,K_{1,n}))$ identified with a compatible pair of $d\times n$ tables $T_0$ and $T_1$, as described in Section 2. By Lemma \ref{Two columns with same rows}, we may assume there exist two columns $c_j, c_{j+1}$ in $T_0$ and their corresponding columns $c_j', c_{j+1}'$ in $T_1$, for some $1\leq  j \leq n$, such that, for each row, $c_j$ has the same entries as $c_j'$, and $c_{j+1}$ has the same entries as $c_{j+1}'$. Note that in Lemma \ref{Two columns with same rows} we use quadratic moves to transform two given tables $T_0$ and $T_1$ into a pair of tables such that they satisfy the condition on columns above. \\
\indent Now, summing coordinatewise the columns $c_j$ and $c_{j+1}$ in $T_0$, the columns $c_j'$ and $c_{j+1}'$ in $T_1$, we obtain a new pair of tables $\hat{T_0}$ and $\hat{T_1}$ with $n -1$ columns. The pair $\hat{T_0}, \hat{T_1}$ is identified with a binomial $\hat{B}\in I(X(G,K_{1,n-1}))$. By definition, this binomial is generated by binomials of degree at most $\phi(G,n-1)$. Hence, we may transform $\hat{T_0}$ into $\hat{T_1}$ by exchanging in every step at most $\phi(G,n-1)$ rows. Each of such steps lifts to an exchange among at most $\phi(G,n-1)$ rows in tables $T_0$ and $T_1$. After applying all the steps, the resulting tables $\tilde T_0$ and $\tilde T_1$ still do not have to be equal. However, they only differ possibly on the columns $c_j$ and $c_{j+1}$. Without loss of generality we may assume $j=1$. Thus the tables $\tilde T_0$ and $\tilde T_1$ are as follows: 

$$
\tilde T_0-\tilde T_1=\begin{bmatrix}
a_{j_1} & b_{j_1} & \ldots & \ldots &\ldots  \\
a_{j_2} & b_{j_2} & \ldots & \ldots & \ldots  \\
\ldots & \ldots & \ldots & \ldots & \ldots \\
a_{j_d} & b_{j_d} & \ldots & \ldots & \ldots\\
\end{bmatrix}
-
\begin{bmatrix}
a_{k_1} & b_{k_1} & \ldots & \ldots & \ldots \\
a_{k_2} & b_{k_2} & \ldots & \ldots & \ldots  \\
\ldots & \ldots & \ldots & \ldots & \ldots \\
a_{k_d} & b_{k_d} & \ldots & \ldots & \ldots \\
\end{bmatrix},
$$

\noindent where columns {\it different} from the first two are identical. Suppose there exists $l$ such that $a_{j_l}\neq a_{k_l}$ (and $b_{j_l}\neq b_{k_l}$). Then $a_{j_l}+ b_{j_l}=a_{k_l} + b_{k_l}$, since the the $l$-th rows of $\tilde T_0$ and $\tilde T_1$ are identical except in the first two columns and, moreover, every row is a flow. On the other hand, there exists $s$ such that $a_{k_l}=a_{j_s}$ and $b_{k_l}=b_{j_s}$. Thus we make a quadratic move between $a_{j_l}, b_{j_l}$ and $a_{j_s}, b_{j_s}$. This concludes the proof. 
\end{proof}
\end{theorem}

\section{Open questions}

In this last section, we collect some well-known open questions regarding group-based models for the convenience of the reader. 
We start from the central conjecture in this context. 

\begin{conjecture}[{\cite[Conjecture 29]{SS}}]
For any finite abelian group $G$, $\phi(G)\leq |G|$.
\end{conjecture}
Taking into account the inductive approach presented in this article, it seems crucial to first understand the simplest tree $K_{1,3}$.
\begin{conjecture}
For any finite abelian group $G$, $\phi(G,3)\leq |G|$.
\end{conjecture}
Notice that our main theorem -- Theorem \ref{finite phylo} -- can be restated as follows: the function $\phi(G,\cdot)$ is eventually constant.
The ensuing result would be a desired strengthening of ours.
\begin{conjecture}[{\cite[Conjecture 9.3]{JaJCTA}}]\label{con:constant}
We have $\phi(G,n+1)=\max(2,\phi(G,n))$.
\end{conjecture}
We are grateful to Seth Sullivant for noticing that this is equivalent to $\phi(G,\cdot)$ being constant, apart from the case when $G=\ZZ_2$ and $n=3$, when the associated variety is the whole projective space.
Conjecture \ref{con:constant} also implies the following.
\begin{conjecture}[{\cite[Conjecture 30]{SS}}]
The phylogenetic complexity of $G=\ZZ_2 \times \ZZ_2$ is $4$.
\end{conjecture}
Yet another direction would be trying to find combinatorial analogs of $\Delta$--modules presented in \cite{MR3018955, MR3430359}. We have not pursued this approach, however we present some similarities. First, in the class of equivariant models one can apply such techniques to prove finiteness on the set--theoretic level \cite{DE}. Second, one of the properties of equivariant models -- a flattening -- is mimicked for group--based models (on the algebra level though, but not on the level of varieties). This is the addition of two group elements that turns a flow of length $n+1$ to a flow of length $n$. The latter was a crucial property that allowed us to obtain the result: generation using the `simple' equations (in our case, quadratic moves) and induced equations for smaller $n$. It would be very desirable to introduce a general setting for polytopes and toric varieties, which would still allow to obtain finiteness results on the ideal-theoretic level.

\vspace{5mm}

\noindent
{\bf Acknowledgements.}\smallskip \\
The first author was supported by Polish National Science Centre grant no. 2015/19/D/ST1/01180 and the Foundation for Polish Science (FNP).
The second author acknowledges fundings from the Doctoral Programme Network at Aalto University.

\bibliographystyle{amsalpha}
{\footnotesize
\bibliography{Xbib}}

\bigskip

\noindent
\footnotesize {\bf Authors' addresses:}

\smallskip

\noindent Mateusz Micha{\l}ek,
Freie Universit\"at Berlin, Germany;
Polish Academy of Sciences, Warsaw, Poland,
{\tt mmichalek@impan.pl}

\smallskip

\noindent Emanuele Ventura,
Aalto University,  Finland,
{\tt emanuele.ventura@aalto.fi}

\end{document}